\documentclass[11pt]{amsart} 
\usepackage{amsfonts, amssymb, enumerate, amsmath, hyperref} 
\usepackage{url} 
\usepackage[all]{xypic}
 
\newtheorem{thm}{Theorem} 
\newtheorem{lem}[thm]{Lemma} 
\newtheorem{cor}[thm]{Corollary} 
\newtheorem{prop}[thm]{Proposition} 

\theoremstyle{definition}
\newtheorem{df}[thm]{Definition} 
\newtheorem{ex}[thm]{Example} 
\newtheorem{rmk}[thm]{Remark} 
 
\begin{document} 
\newcommand         {\rar}[1]       {\stackrel{#1}{\longrightarrow}} 
 
\def    \onto   {\twoheadrightarrow} 
 
\newcommand         {\har}[1]       {\stackrel{#1}{\hookrightarrow}}

\newcommand\rank{\operatorname{rank}\,} 
\newcommand\rdim{\operatorname{rdim}} 
\newcommand\Hom{\operatorname{Hom}} 
\newcommand\mdim{\operatorname{mdim}} 
\newcommand\Ind{\operatorname{Ind}} 
\newcommand\Id{\operatorname{Id}} 
\newcommand\Tr{\operatorname{Tr}} 
\newcommand\F{\operatorname{\mathbb{F}}} 
\newcommand\M{\operatorname{\mathbb{M}}} 
\newcommand\Z{\operatorname{\mathbb{Z}}} 
\newcommand\C{\operatorname{\mathbb{C}}} 
\def \F  {\mathbb{F}} 
 
\def    \ind    {\mathrm{ind}} 
 
\def    \tpsi   {\tilde{\psi}} 
 
 \def	\cA	{\mathcal{A}}
 \def	\bZ	{\mathbb{Z}}
 \def	\bZp	{\bZ/{p\bZ}}
 \def	\bZt	{\bZ/{2\bZ}}
 
\newcommand\GL{\operatorname{GL}} 
\newcommand\Mat{\operatorname{M}} 
\newcommand\End{\operatorname{End}} 
\newcommand\Core{\operatorname{Core}} 
\newcommand\Ker{\operatorname{Ker}} 
\newcommand     \ra     {\rightarrow} 
\newcommand     \inj    {\hookrightarrow} 
\newcommand         {\isom}         {\rar{\simeq}}

\title{Maximal representation dimension of finite $p$-groups}

\author{Shane Cernele, Masoud Kamgarpour, and Zinovy Reichstein} 
\thanks{The authors gratefully acknowledge financial support
by the Natural Sciences and Engineering Research Council of Canada}
 
\address{Department of Mathematics, University of British Columbia, 
Vancouver, BC V6T 1Z2, Canada} 
\address{scernele@math.ubc.ca, masoud@math.ubc.ca, reichst@math.ubc.ca} 
 
\begin{abstract} 
The representation dimension $\rdim(G)$ of a finite group $G$ is 
the smallest positive integer $m$ for which there exists an embedding 
of $G$ in $\GL_m(\C)$. In this paper we find 
the largest value of $\rdim(G)$, as $G$ ranges over all groups 
of order $p^n$, for a fixed prime $p$ and a fixed exponent $n \ge 1$.\\ 
\end{abstract} 


 
\maketitle 
 
\section{Introduction}\label{s:intro} 
 
The representation dimension of a finite group $G$, denoted 
by $\rdim (G)$, is the minimal dimension of a faithful 
complex linear representation of $G$. In this paper we determine 
the maximal representation dimension of a group of 
order $p^{n}$. We are motivated by a recent result of N. Karpenko and 
A. Merkurjev \cite[Theorem 4.1]{km}, which states that 
if $G$ is a finite $p$-group then the essential 
dimension of $G$ is equal to $\rdim (G)$. For a detailed 
discussion of the notion of essential dimension for finite 
groups (which will not be used in this paper),
see \cite{br} or \cite[\S 8]{jly}. 
We also note that a related invariant, the
minimal dimension of a faithful complex 
{\em projective} representation of $G$, has been extensively 
studied for finite simple groups $G$; for an overview, 
see~\cite[\S 3]{tz}.
 
Let $G$ be a $p$-group of order $p^n$ and $r$ be the rank of 
the centre $Z(G)$.
A representation of $G$ is faithful 
if and only if its restriction to $Z(G)$ is faithful. 
Using this fact it is easy to see that a faithful representation
$\rho$ of $G$ of minimal dimension decomposes as a direct sum
\begin{equation} \label{e.decomp}
\rho = \rho_1 \oplus \dots \oplus \rho_r
\end{equation}
of exactly $r$ irreducibles; cf. ~\cite[Theorem 1.2]{mr}.
Since the dimension of any irreducible representation 
of $G$ is $\le \sqrt{[G:Z(G)]}$ (see, e.g.,~\cite[Corollary 3.11]{W03})
and $|Z(G)| \ge p^r$, we conclude that
\begin{equation} \label{e.inequality}
\rdim(G) \leq rp^{\left\lfloor(n-r)/2\right\rfloor} .
\end{equation}
Let 
\[ 
f_p(n) := \max_{r\in\mathbb{N}}(rp^{\left\lfloor(n-r)/2\right\rfloor}). 
\] 
It is easy to check that $f_p(n)$ is given by the following table: 
 
\begin{center} 
\vspace{0.1cm} 
\begin{tabular}{|c| c| c|} 
 
\hline 
$n$&$p$ & $f_p(n)$\\ 
\hline 
even & arbitrary & $2p^{(n-2)/2}$\\ 
\hline 
odd & odd & $p^{(n-1)/2}$\\ 
\hline 
odd, $\ge 3$ & $2$ & $3p^{(n-3)/2}$\\ 
\hline 
$1$ & $2$ & $1$ \\ 
\hline 
\end{tabular} 
\end{center} 

We are now ready to state the main result of this paper.
 
\begin{thm} \label{main} Let $p$ be a prime and $n$ be a positive integer.
For almost all pairs $(p,n)$, the maximal value of
$\rdim(G)$, as $G$ ranges over all groups of order $p^n$,
equals $f_p(n)$.  The exceptional cases are 
\[ 
\text{$(p,n) = (2,5)$, $(2,7)$ and $(p, 4)$, where $p$ is odd.}
\]
In these cases the maximal representation 
dimension is $5$, $10$, and $p + 1$, respectively. 
\end{thm} 

The proof will show that the maximal value of $\rdim(G)$,
as $G$ ranges over all groups of order $p^n$, is
always attained for a group $G$ of nilpotency class $\le 2$.
Moreover, if $(p, n)$ is non-exceptional, $n \ge 3$ and
$(p, n) \neq (2, 3), (2, 4)$, 
the maximum is attained on a special class of $p$-groups 
of nilpotency class $2$. We call these groups 
{\em generalized Heisenberg groups} since their 
representation theory looks very similar to the usual 
Heisenberg group (the group of unipotent upper triangular 
$3\times 3$ matrices); see Section~\ref{sect.rep}

The rest of this paper is structured as follows.
In \S \ref{s:heisenberg} we introduce 
generalized Heisenberg groups and study their irreducible 
representations. In \S \ref{s:proof}, we prove  Theorem~\ref{main}. 
 
\subsection*{Acknowledgement} 
 We would like to thank Hannah Cairns, Robert 
Guralnick, Chris Parker, Burt Totaro, and Robert 
Wilson for helpful discussions. We are also grateful to
the referee for constructive comments. 
 
\section{Generalized Heisenberg groups}\label{s:heisenberg} 
\subsection{Spaces of alternating forms}  

Let $V$ be a finite dimensional vector space over an 
arbitrary field $F$. Let $\cA(V)$ denote the space of bilinear 
alternating forms on $V$; that is, linear maps 
$b:V\otimes V\ra F$ satisfying $b(v,v)=0$.  

Let $K$ be a subspace of $\cA(V)$. Then $K$ defines 
a map $\omega_K: V\times V\ra K^*$ as follows. 
Let $j: \cA(V)^*\ra K^*$ denote the dual of 
the natural injection $K\inj \cA(V)$.  
Then  $\omega_K$ is defined to be the composition 
\begin{equation}\label{eq:omega}
\xymatrix{
V \times V \ar[r] \ar@/ _1.7pc/[rrr]_{\omega_K}  & \Lambda^{2}(V)\ar[r]  
& \cA(V)^*\ar[r]^j  &K^{*},\\\\
}
\end{equation}
where the first map is the natural projection and the second one is 
the canonical identification of the two spaces. 

\subsection{Symplectic subspaces} 
\begin{df} A subspace $K\subseteq \cA(V)$ is {\em symplectic} if 
every nonzero element of $K$ is non-degenerate, as a bilinear form on $V$. 
\end{df}

\begin{rmk} \label{r:symplectic}
Equivalently, $K\subset \cA(V)$ is symplectic 
if and only if for every nonzero linear map 
$K^*\ra F$ the composition $V\times V\rar{\omega_K} K^* \ra F$ 
is non-degenerate. 
\end{rmk}

 Clearly nontrivial symplectic 
subspaces of $\cA(V)$ can exist only if $\dim(V)$ 
is even. 

\begin{lem} \label{l:Existence}
Suppose $V$ is an $F$-vector space of dimension $2m$.
If $F$ admits a field extension of degree $m$ then there exists 
an $m$-dimensional symplectic subspace $K\subset \cA$.
\end{lem}

\begin{proof} 
Choosing a basis of $V$, we can identify $\cA(V)$ 
with the space of alternating $2m \times 2m$-matrices.
Let $f \colon \Mat_m(F) \to \cA(V)$ be the
linear map
\[ A \mapsto \begin{bmatrix} 0 & A\\ -A^{T} & 0\\ \end{bmatrix} \, . \]
If $W$ is a linear subspace of $\Mat_m (F) = \End_F(F^m)$ 
such that $W \backslash \{0\} \subset {\rm GL}_m(F)$ then 
$K = f(W)$ is a symplectic subspace.

It thus remains to construct an $m$-dimensional 
linear subspace $W$ of $\Mat_m(F)$ 
such that $W \backslash \{0\} \subset {\rm GL}_m(F)$.
Let $E$ be a degree $m$ field extension of $F$. Then $E$ acts 
on itself by left multiplication. This gives an $F$-vector 
space embedding of $\Psi \colon E \hookrightarrow \End_F(E)$ such that 
$\Psi(e)$ is invertible for all $e \neq 0$.
\end{proof}

\subsection{Groups associated to spaces of alternating forms}
\label{ss:HeisenbergGrps}

Let $V$ be a finite-dimensional vector space over a field $F$. 
Let $K$ be a subspace of $\cA(V)$ and let $\omega_K$ denote 
the induced map $V\times V\ra K^*$, see (\ref{eq:omega}). 
Choose a bilinear map $\beta:V\times V\ra K^*$ such that 
\begin{equation}\label{eq:decomposition} 
\omega_K(v,w)=\beta(v,w)-\beta(w,v). 
\end{equation}
To see that this can always be done, note that 
if $\{e_i\}$ is a basis of $V$, we can define $\beta$ by
\[ \beta(e_i, e_j) = \begin{cases} 
\text{$\omega_K(e_i,e_j)$, if $i>j$ and} \\
\text{$0$, otherwise.} 
\end{cases} \] 
We also remark that $\beta$ is uniquely determined by $\omega_K$,
up to adding a symmetric bilinear form $V \times V \to K^*$.

\begin{df} \label{def.group}
Let $H=H(V,K,\beta)$ denote the group whose underlying set is
$V\times K^*$ and whose multiplication is given by 
\begin{equation} 
(v,t) \cdot(v',t')=(v+v',t+t'+\beta(v,v')). \label{eq:HeisGroup} 
\end{equation} 
If $K$ is a symplectic subspace, we will refer to $H$ as a
{\em generalized Heisenberg group}. 
\end{df}

\begin{ex} \label{ex:Heisenberg} Suppose $\omega$ is a nondegenerate alternating bilinear form on
 $V=F\oplus F$, where $F$ is a field of characteristic not equal to $2$. Let $K$ be the span of $\omega$ in $A(V)$. Then
$H(V,K, \frac{1}{2}\omega)$ is isomorphic to the group of unipotent upper triangular $3\times 3$ matrices over $F$. This group is known as the Heisenberg group.  
\end{ex}

\begin{rmk} \label{rem.centre}It is easy 
to see that~\eqref{eq:HeisGroup} is indeed a group law with
the inverse given by $(v, t)^{-1} = (-v, - t + \beta(v, v))$ and
the commutator given by
\begin{equation} \label{e.commutator}
[(v_1, t_1), (v_2, t_2)] = (0, \omega_K(v_1, v_2)) \, .
\end{equation}
As  $\omega_K$ is surjective, we see that $[H, H] = K^*$. Moreover,~\eqref{e.commutator} also shows that $K^* \subset Z(H)$, 
and that equality 
holds unless the intersection $\displaystyle \cap_{k \in K} \ker(k)$ is nontrivial.
In particular, $Z(H) = K^*$ if $K$ contains a symplectic form.
\end{rmk}
 
\begin{rmk} \label{rem.special} A non-abelian finite $p$-group $S$ is 
called {\em special} if $Z(S)=[S,S]$ and $S/[S,S]$ is elementary abelian; see~\cite[\S 2.3]{hall}. 
Suppose K is a subspace of $\cA(V)$ such 
that $\displaystyle \cap_{k \in K} \ker(k)$ is trivial. Then over 
the finite field $\F_p$, the groups $H(V,K,\beta)$ are examples 
of non-abelian special $p$-groups.
We are grateful to the referee for pointing this out. 
\end{rmk}

\begin{rmk}\label{r:uniqueness}
If $\beta$ and $\beta'$ both satisfy~\eqref{eq:decomposition} 
then $H(V, K, \beta)$ may not be isomorphic to $H(V, K, \beta')$. 
For example, let $V$ be a 2-dimensional vector space over $F=\F_2$, 
$K$ be the one-dimensional (symplectic) subspace generated by  
$\begin{bmatrix} 
  0 & 1\\ 
  1 & 0\\ 
\end{bmatrix}$,  
and $\beta$, $\beta'$ be bilinear forms on $V$ defined by $\begin{bmatrix} 
  1 & 1\\ 
  0 & 1\\ 
\end{bmatrix}$ and 
$\begin{bmatrix} 
  0 & 1\\ 
  0 & 0\\ 
\end{bmatrix}$, respectively. 
Then $\beta$ and $\beta'$ both satisfy \eqref{eq:decomposition}, 
but $H(V,K,\beta)$ is isomorphic to the quaternion group 
while $H(V,K,\beta')$ is isomorphic to the dihedral group of order $8$. 

On the other hand, it is easy to see that $H(V,K,\beta)$ and $H(V,K,\beta')$ 
are always isoclinic. 
(Two groups $S$ and $T$ are isoclinic if there are 
isomorphisms $f: S/Z(S)\rightarrow T/Z(T)$ 
and $g:[S,S]\rightarrow  [T,T]$ such that 
if $a, b \in S$ and $a', b' \in T$ 
with $f(aZ(S))=a'Z(T)$ and $f(bZ(S))=b'Z(T)$, 
then we have $g([a,b])=[a', b']$, see \cite{Ha40}.)
\end{rmk}

\subsection{Representations}
\label{sect.rep}

Let $p$ be an arbitrary prime and let $F=\F_p$ be the finite field 
of $p$ elements. Fix, once and for all, a homomorphism 
$\tau : (\F_p, +) \inj \C^{*}$. Let $W$ be a vector space 
over $F$. Using $\tau$, we identify the algebraic 
dual $W^*=\Hom(W,F)$ with the Pontriyagin dual 
$\Hom(W,\C^{*})$. It is clear that a bilinear alternating 
map $W\times W\ra \F_p$ is non-degenerate if and only 
if the composition $W\times W\ra \F_p\rar{\tau} \C^\times$ 
is non-degenerate.

Now let $V$ be a vector space over $F$, $K$ a subspace of $\cA(V)$, 
and $\omega=\omega_K$ the associated map. 
Choose $\beta$ satisfying (\ref{eq:decomposition}) 
and let $G = H(V, K, \beta) = V \times K^*$. Recall 
that $K^*$ is in the center of $G$ (Remark \ref{rem.centre}); 
in particular, it acts via a character on every irreducible 
representation of $G$. 
 
\begin{lem} \label{l:uniqueird} Let $\rho$ be an irreducible 
representation of $G$ such that $K^{*}$ acts by $\psi$. 
Assume $\psi \circ \omega: V\times V \ra \C^{\times}$ 
is non-degenerate.
\begin{enumerate} 
\item[(a)]If $g \in G$, $g \notin K^{*}$, then $\Tr(\rho(g))=0$. 
\item[(b)]$\dim(\rho)=\sqrt{|V|}$. 
\item[(c)]$\rho$ is uniquely determined (up to isomorphism) by $\psi$. 
\end{enumerate} 
\end{lem} 

\begin{proof} 
(a) Let $g \in G\backslash K^*$. 
Since $\psi \circ \omega$ is non-degenerate 
there exists $h \in G$ such that 
$\psi \circ \omega(gK^{*},hK^{*}) \neq 1$. 
Observe that $\rho([g,h]) = \psi([g,h])\Id$, 
and that $\rho(h^{-1}gh) = \rho(g)\rho([g,h])$. 
Taking the trace of both sides, 
we have $\Tr(\rho(g)) = \psi([g,h])\Tr(\rho(g))$. 
Since $\psi([g,h]) \neq 1$ we must have $\Tr(\rho(g))=0$. 
 
\smallskip
(b) Since $\rho$ is irreducible, and the trace of $\rho$ 
vanishes outside of $K^{*}$, we have: 
\begin{eqnarray*} 1 &=& \frac{1}{|G|}\sum_{g \in G} 
\Tr(\rho(g)) \overline{\Tr(\rho(g))} \\ 
&=&\frac{1}{|G|}\sum_{g \in K^{*}}\Tr(\rho(g))\overline{\Tr(\rho(g))} \\ 
&=& \frac{1}{|G|}\dim(\rho)^{2}\sum_{g \in K^{*}}\Tr(\psi(g))
\overline{\Tr(\psi(g))}\\ 
&=&\dim(\rho)^{2}\frac{|K^{*}|}{|G|} 
\end{eqnarray*} 
Thus $\dim \rho=\sqrt{|G|/|K^{*}|} = \sqrt{|V|}$.

\smallskip
(c) We have completely described the character of $\rho$, and it follows that $\rho$ is uniquely determined by $\psi$. Indeed, 
\[
\Tr(\rho(g)) = \begin{cases}
\sqrt{|V|} \cdot \psi(g), & \text{if $g \in K^{*}$ and} \\
0 & \text{otherwise.}\\
\end{cases}
\]
\end{proof} 

In view of Remark \ref{r:symplectic}, the following proposition 
is a direct consequence of the above lemma. 
 
\begin{prop} \label{p:repHeisenberg} The irreducible representations 
of a generalized Heisenberg group $H=H(V,K,\beta)$ are exhausted by the following list: 
\begin{enumerate} 
 \item[(i)] $|V|$ one-dimensional representations, 
one for every character of $V$. 
 \item[(ii)] $|K|-1$ representations of 
dimension $\sqrt{|V|}$, one for every nontrivial 
character $\psi:K^*\ra \C^\times $. 
\end{enumerate} 
\end{prop} 
 
The next corollary is also immediate upon observing the centre of a generalized Heisenberg groups $H=H(G,K,\beta)$
equals $K^{*}$; see Remark~\ref{rem.centre}.
 
\begin{cor} \label{c:RepDimGenHeis} 
The representation dimension of a generalized Heisenberg group $H=H(V,K,\beta)$ equals $\dim(K)\sqrt{|V|}$. 
\end{cor} 

If $G$ is a finite Heisenberg group in the usual sense 
(as in Example~\ref{ex:Heisenberg}) then for each nontrivial 
character $\chi$ of $Z(G)$ there is a unique irreducible 
representation $\psi$ of $G$ whose central character is $\chi$; 
cf.~\cite[\S1.1]{GH07}.  This is a finite group variant of the celebrated
Stone-von Neumann Theorem. For a detailed discussion of the history
and the various forms of the Stone-von Neumann theorem we refer 
the reader to~\cite{rosenberg}. We conclude this section with
another immediate corollary
of Proposition~\ref{p:repHeisenberg} which tells us that over 
the field $\mathbb{F}_p$ every generalized Heisenberg group
has the Stone-von Neumann property. 
This corollary will not be needed in the sequel.

\begin{cor} \label{cor.stone} Two irreducible representations 
of a generalized Heisenberg group with the same nontrivial 
central character are isomorphic. 
\end{cor} 

Corollary~\ref{cor.stone} is the reason we chose to use the term
``generalized Heisenberg group" in reference to the groups 
$H(V, K, \beta)$, where K is a symplectic subspace. 
Special $p$-groups (Remark~\ref{rem.special}) which are 
not generalized Heisenberg groups may not have the Stone-von 
Neumann property; see Remark \ref{r:SpecialNotHeisenberg}.

\section{Proof of Theorem \ref{main}} \label{s:proof} 
 
The case where $n \le 2$ is trivial; clearly $\rdim(G) = \rank(G)$ 
if $G$ is abelian. We will thus assume that $n \ge 3$. 

In the non-exceptional cases of the theorem, 
in view of the inequality~\eqref{e.inequality},
it suffices to construct 
a group $G$ of order $p^n$ with $\rdim(G) = f_p(n)$. Here
$f_p(n)$ is the function defined just before 
the statement of Theorem~\ref{main}. 

If $(p, n) = (2, 3)$ or $(2, 4)$, we take $G$ to be the elementary 
abelian group $(\bZt)^3$ and $(\bZt)^4$, yielding
the desired representation dimension of $3$ and $4$, respectively. 
For all other non-exceptional pairs $(p, n)$, we take $G$ to be a generalized 
Heisenberg group as described in the table below.
Here $H(V, K)$ stands for $H(V, K, \beta)$, 
for some $\beta$ as in~\eqref{eq:decomposition}.
In each instance, the existence of a symplectic subspace $K$ of suitable 
dimension is guaranteed by Lemma \ref{l:Existence}  and the value of
$\rdim(H(V, K))$ is given by Corollary~\ref{c:RepDimGenHeis}. 
 
\begin{center} 
\vspace{0.1cm} 
\begin{tabular}{|c| c| c| c| c|} 
 
\hline 
$n$   &   $p$ &    $\dim(V)$ & $\dim(K)$ & $\rdim(H(V,K))$\\ 
\hline 
even, $\ge 6$ & arbitrary & $n-2$ &  2   & $2p^{(n-2)/2}$\\ 
\hline 
odd, $\ge 3$ &   odd &   $n-1$ & 1 & $p^{(n-1)/2}$\\ 
\hline 
odd, $\ge 9$ & $2$ & $n-3$ & 3 & $3p^{(n-3)/2}$\\ 
\hline 
\end{tabular} 
\end{center} 
 
\smallskip
This settles the generic case of Theorem~\ref{main}.
We now turn our attention to the exceptional cases. 
We will need the following upper bound on $\rdim(G)$, 
strengthening~\eqref{e.inequality}.

Let $\Omega_{1}(Z(G))$ be the subgroup of elements 
$g \in Z(G)$ such that $g^p = 1$.  

\begin{lem} \label{lem.inequality2}
Let $G$ be a $p$-group and 
$r = \rank(Z(G)) = \rank(\Omega_{1}(Z(G)))$.

\smallskip
(a) Let $\rho_1$ be an irreducible representation of $G$
such that $\Ker(\rho_1)$ does not contain $\Omega_{1}(Z(G))$. Then
there are irreducible representations $\rho_2, \dots, \rho_r$ of $G$
such that $\rho_1 \oplus \dots \oplus \rho_r$ is faithful. In particular, 
\[ \rdim(G) \le \dim(\rho_1) + (r-1) \sqrt{[G:Z(G)]} \, . \] 

(b) If $\Omega_{1}(Z(G))$ is not contained in $[G, G]$,
then 
\[ \rdim(G) \le 1 + (r-1) \sqrt{[G:Z(G)]} \, . \] 
\end{lem}

The lemma can be deduced from~\cite[Remark 4.7]{km}
or~\cite[Theorem 1.2]{mr}; for the sake of completeness we give 
a self-contained proof.

\begin{proof}
(a) Let $\chi_1$ be the restriction to $\Omega_{1}(Z(G))$ of the central 
character of $\rho_1$. By our assumption $\chi_1$ is nontrivial.
Complete $\chi_1$ to a basis $\chi_1, \chi_2, \dots, \chi_r$ 
of the $r$-dimensional $\F_p$-vector space $\Omega_{1}(Z(G))^*$ 
and choose an irreducible representation $\rho_i$ such that 
$\Omega_1(Z(G))$ acts by $\chi_i$. (The representation
$\rho_i$ can be taken to be any irreducible component
of the induced representation $\Ind_{\Omega_{1}(Z(G))}^G(\chi)$.)
The restriction 
of $\rho \colon = \rho_1 \oplus \dots \oplus \rho_r$ 
to $\Omega_{1}(Z(G))$ is faithful. Hence, $\rho$ is a faithful
representation of $G$. As we mentioned in the introduction
$\dim(\rho_i) \le \sqrt{[G: Z(G)]}$ for every $i \ge 2$, and 
part (a) follows.

\smallskip
(b) By our assumption there exists a one-dimensional representation
$\rho_1$ of $G$ whose restriction to $\Omega_{1}(Z(G))$ is nontrivial.
Now apply part (a).
\end{proof}

We are now ready to prove Theorem~\ref{main} in the three exceptional cases.
 
\subsection{Exceptional case 1: $p$ is odd and $n=4$} 
 
\begin{lem} \label{lem.p^4}
Let $p$ be an odd prime and $G$ be a group of order $p^4$. 

\smallskip
(a) Then $\rdim(G)\leq p+1$. 

\smallskip
(b) Suppose $Z(G) \simeq (\bZp)^2$ and $G/Z(G) \simeq (\bZp)^2$. 
Then $\rdim(G) = p + 1$.
\end{lem} 

\begin{proof} (a) 
We argue by contradiction. Assume there exists a group of order $p^4$ such that
$\rdim(G) \ge  p+2$. If $|Z(G)| \ge p^3$ or $G/Z(G)$ is cyclic
then $G$ is abelian and $\rdim(G) = \rank(G) \le 4 \le p+1$, a contradiction.
If $Z(G)$ is cyclic then $\rdim(G) \le p$ by~\eqref{e.inequality}, again
a contradiction. 

Thus $Z(G) \simeq G/Z(G) \simeq (\bZp)^2$. 
This reduces part (a) to part (b).

\smallskip
(b) Here $\Omega_{1}(Z(G)) = Z(G)$ has rank $2$. Hence,
a faithful representation $\rho$ of $G$
of minimal dimension is the sum of two irreducibles
$\rho_1 \oplus \rho_2$, as in~\eqref{e.decomp},
each of dimension $1$ or $p$.

Clearly $\dim(\rho_1) = \dim(\rho_2) = 1$ is not possible, since
in this case $G$ would be abelian, contradicting $[G:Z(G)] = p^2$.
It thus remains to show that $\rdim(G) \le p + 1$.
Since $G/Z(G)$ is abelian, $[G,G] \subset Z(G)$. 
Hence, by Lemma~\ref{lem.inequality2}(b) we only need to establish 
that $[G, G] \subsetneq Z(G)$. 

To show that $[G, G] \subsetneq Z(G)$, note that the commutator map 
\begin{eqnarray*} \Psi: G/Z(G) \times G/Z(G) &\rightarrow& [G,G]\\ 
(gZ(G), g'Z(G)) &\mapsto& [g,g']
\end{eqnarray*} 
can be thought of as an alternating bilinear map from 
$\F_p^{2}$ to itself. Viewed in this way, $\Psi$ can 
be written as $\Psi(v,v')=(w_1(v,v'), w_2(v,v'))$ for 
alternating maps $w_1$ and $w_2$ from $(\F_p)^{2}$ to $\F_p$. 
Since the space of alternating maps is a one-dimensional vector space
over $\F_p$, $w_1$ and $w_2$ are scalar multiples of each other. 
Hence, the image of $\Psi$ is a cyclic group of order $p$, and
$[G,G] \subsetneq Z(G)$, as claimed. 
\end{proof} 

To finish the proof of Theorem~\ref{main} in this case, note that
$G = \bZp \times G_0$, where $G_0$ is 
a non-abelian group of order $p^3$, satisfies the conditions
of Lemma~\ref{lem.p^4}(b). Thus the maximal representation dimension 
of a group of order $p^{4}$ is $p+1$, for any odd prime $p$.

\subsection{Exceptional case 2: $p=2$ and $n=5$} 
\begin{lem} Let $G$ be a group of order $32$. Then $\rdim(G)\leq 5$. 
\end{lem} 

\begin{proof}  We argue by contradiction. Assume there exists a group 
of order $32$ and representation dimension $\ge 6$. Let
$r = \rank(Z(G))$. Then $1 \le r \le 5$ and
\eqref{e.inequality} shows that $\rdim(G) \le 5$
for every $r \neq 3$.

Thus we may assume $r = 3$.
If $|Z(G)| \ge 16$ or $G/Z(G)$ is cyclic
then $G$ is abelian, and $\rdim(G) = \rank(G) \le 5$. 
We conclude that $Z(G) \simeq
(\bZt)^3$ and $G/Z(G) \simeq (\bZt)^2$.  Applying the same 
argument as in the proof of Lemma~\ref{lem.p^4}(b),
we see that $[G,G] \subsetneq Z(G)$, and hence 
$\rdim (G) \leq 5$ by Lemma~\ref{lem.inequality2}(b), a contradiction.
\end{proof} 

To finish the proof of Theorem~\ref{main} in this case, note that 
the elementary abelian group of order $2^{5}$ has representation 
dimension $5$. Thus the maximal representation dimension of 
a group of order $2^{5}$ is $5$.
 
\subsection{Exceptional case 3: $p=2$ and $n=7$} 
\begin{lem} \label{l:rdimleqten}If $|G| = 128$ then $\rdim(G) \le 10$. 
\end{lem} 

\begin{proof} 
Again, we argue by contradiction. Assume there exists a group $G$ of order $128$
and representation dimension $\ge 11$. Let $r$ be the rank of $Z(G)$. 
By~\eqref{e.inequality}, $r = 3$; otherwise we would have $\rdim(G) \le 10$.

As we explained in the introduction, 
this implies that a faithful representation $\rho$ of $G$ of minimal dimension
is the direct sum of three irreducibles $\rho_1$, $\rho_2$
and $\rho_3$, each of dimension $\le \sqrt{2^7/|Z(G)|}$.
If $|Z(G)| > 8$, then $\dim(\rho_i) \le 2$ and 
$\rdim(G) = \dim(\rho_1) + \dim(\rho_2) + \dim(\rho_3) \le 6$, a
contradiction.  

Therefore, $Z(G) \cong (\bZt)^{3}$ and 
$\dim(\rho_1) = \dim(\rho_2) = \dim(\rho_3) = 4$.
By Lemma~\ref{lem.inequality2}(a) this implies that the kernel of 
every irreducible representation of $G$ of dimension $1$ or 
$2$ must contain $Z(G)$. In other words, any such representation
factors through the group $G/Z(G)$ of order $16$. 
Consequently, if $m_i$ is the number of irreducible 
representations of $G$ of dimension $i$ then
$m_1 + 4 m_2 = 16$. We can now appeal 
to~\cite[Tables I and II]{JNO90}, to show that no group of 
order $2^{7}$ has these properties.  From Table I 
we can determine which groups $G$ (up to isoclinism, 
cf.~Remark~\ref{r:uniqueness}) 
have $|Z(G)|=8$ and using Table II we can determine $m_1$ and $m_2$ 
for these groups. There is no group $G$ 
with $|Z(G)|=8$ and $m_1 + 4m_2 = 16$. 
\end{proof} 
 
We will now construct an example of a group $G$ of order $2^{7}$ 
with $\rdim(G)=10$. Let $V = (\F_2)^4$ and
let $K$ be the 3-dimensional subspace 
of $A(V)$ generated by the following three elements: 
\[ 
 \begin{bmatrix} 
  0 & 0 & 0 & 1\\ 
  0 & 0 & 1 & 0\\ 
  0 & 1 & 0 & 0\\ 
  1 & 0 & 0 & 0\\ 
\end{bmatrix} 
, \quad 
\begin{bmatrix} 
  0 & 0 & 1 & 0\\ 
  0 & 0 & 1 & 1\\ 
  1 & 1 & 0 & 0\\ 
  0 & 1 & 0 & 0\\ 
\end{bmatrix} 
, \quad
\begin{bmatrix} 
  0 & 0 & 1 & 1\\ 
  0 & 0 & 0 & 1\\ 
  1 & 0 & 0 & 1\\ 
  1 & 1 & 1 & 0\\ 
\end{bmatrix} \, . 
\] 
Let $G:=H(V,K, \beta)=V \times K^{*}$ for some $\beta$ 
as in~\eqref{eq:decomposition}.
Note that $K$ contains only one non-zero degenerate 
element (the sum of the three generators). 
In other words, there is only one non-trivial 
character $\chi$ of $K^{*}$ such that 
$\chi \circ \omega: V \times V \rightarrow \C^{\times}$ 
is degenerate. By Remark~\ref{rem.centre} 
\begin{equation}\label{eq:HSpecial}
 [G,G] = Z(G) = K^{*}.
 \end{equation}
Let $\rho$ be a faithful representation of $G$ of minimal dimension. 
As we explained in the Introduction, 
$\rho$ is the sum of $\rank(Z(G)) = 3$ irreducibles. Denote 
them by $\rho_1$, $\rho_2$, and $\rho_3$, and their central characters
by $\chi_1$, $\chi_2$ and $\chi_3$, respectively. Since
$\rho$ is faithful, $\chi_1$, $\chi_2$ and $\chi_3$
form an $\mathbb F_2$-basis of $\Omega_{1}(Z(G))^{*} \simeq (\bZt)^3$.
By Lemma \ref{l:uniqueird}, for each 
nontrivial character $\chi$ of $K^{*}$ except one, there 
is a unique irreducible representation $\psi$ of $G$ such 
that $\chi$ is the central character 
to $\psi$, and $\dim \psi =4$. 
Thus at least $2$ of the irreducible components of $\rho$,
say, $\rho_1$ and $\rho_2$ must have dimension $4$. 
By Lemma \ref{l:rdimleqten}, $\dim(\rho) \le 10$, i.e., $\dim(\rho_3) \le 2$.
But every one-dimensional representation of $G$ has trivial 
central character. We conclude that $\dim(\rho_3) = 2$
and consequently $\rdim(G)= \dim(\rho) = 4+4+2=10$. 

Thus the maximal representation dimension of a group of order $2^{7}$ is $10$.

\begin{rmk} \label{r:SpecialNotHeisenberg}
The group $G$ constructed above has $16$ one-dimensional 
representations with trivial central character, 
$4$ two-dimensional representations with non-trivial 
degenerate central character, and $6$ four-dimensional 
representations with pair-wise distinct 
non-degenerate central characters. 
In view of (\ref{eq:HSpecial}), $H$ is a non-abelian 
special 2-group which does not enjoy the Stone-Von Neumann 
property (Corollary \ref{cor.stone}).  
\end{rmk}

\end{document}